\documentclass[12pt,thmsa]{amsproc}
\usepackage[dvips]{graphics}
\input{amssym.def}
\input{amssym.tex}

\def\gnk{G_{n,k}}
\def\cgnk{\Pi_{n,k}}

\def\Cal{\mathcal}

\def\M{{\Cal M}}

\def\D{{\Cal D}}

\def\gnk{G_{n,k}}

\def\bbm{{\Bbb M}}
\def\bbr{{\Bbb R}}

\def\bbn{{\Bbb N}}
\def\bbh{{\Bbb H}}

\def\bbe{{\Bbb E}}
\def\bbs{{\Bbb S}}

\def\dist{{\hbox{\rm dist}}}

\def\sgn{{\hbox{\rm sgn}}}

\def\const{{\hbox{\rm const}}}

\def\gnk{G_{n,k}}

\def\rn{\bbr^n}

\def\part{\partial}
\def\intl{\int\limits}
\def\b{\beta}

\def\Gam{\Gamma}

\def\a{\alpha}
\def\om{\omega}

\def\del{\delta}
\def\vp{\varphi}

\def\gam{\gamma}

\def\sig{\sigma}
\def\lam{\lambda}
\def\z{\zeta}

\def\t{\tau}

\newtheorem{theorem}{Theorem}[section]
\newtheorem{lemma}[theorem]{Lemma}

\theoremstyle{definition}

\newtheorem{example}[theorem]{Example}

\theoremstyle{remark}
\newtheorem{remark}[theorem]{Remark}

\theoremstyle{corollary}

\newtheorem{proposition}[theorem]{Proposition}

\numberwithin{equation}{section}

\newcommand{\be}{\begin{equation}}
\newcommand{\ee}{\end{equation}}

\newcommand{\bea}{\begin{eqnarray}}
\newcommand{\eea}{\end{eqnarray}}
\newcommand{\Bea}{\begin{eqnarray*}}
\newcommand{\Eea}{\end{eqnarray*}}

\def\sideremark#1{\ifvmode\leavevmode\fi\vadjust{\vbox to0pt{\vss
 \hbox to 0pt{\hskip\hsize\hskip1em
\vbox{\hsize2cm\tiny\raggedright\pretolerance10000
 \noindent #1\hfill}\hss}\vbox to8pt{\vfil}\vss}}}%

\begin{document}

\title[On the Funk-Radon-Helgason Inversion Method]
{On the Funk-Radon-Helgason Inversion Method in Integral Geometry}

\author{Boris Rubin}
\address{
Department of Mathematics, Louisiana State University, Baton Rouge,
LA, 70803 USA}

\email{borisr@math.lsu.edu}

\subjclass[2000]{Primary 44A12; Secondary 47G10}


\dedicatory{Dedicated to Professor Sigurdur Helgason
on the occasion of his 85th birthday}

\keywords{Radon  transforms, Inversion Formulas,  $L^p$ spaces.}

\begin{abstract} The  paper deals with totally geodesic Radon transforms on constant curvature spaces.
We study applicability of the historically the first Funk-Radon-Helgason  method of mean value operators to reconstruction of  continuous and  $L^p$ functions from their Radon transforms. New inversion formulas involving Erd\'elyi-Kober type fractional integrals are obtained. Particular   emphasis is placed on the choice of the differentiation operator in the spirit of the recent Helgason's formula.
\end{abstract}

\maketitle

\section{Introduction}

\setcounter{equation}{0}

Inversion of Radon transforms is one of the central topics of integral geometry  \cite{GGG, GGV, He}. The method of mean value operators,
 when the unknown function is reconstructed from its spherical mean,  was suggested by Funk \cite{F11, F13}  for circular transforms on the $2$-sphere. It was adapted by Radon \cite{Rad} for hyperplane transforms, and extended by Helgason \cite{He} to totally geodesic transforms on arbitrary constant curvature space in any dimension.

 In most  publications the method of mean value operators is applied to infinitely differentiable rapidly decreasing functions.
 Below we investigate applicability of this method   to arbitrary continuous
 and $L^p$ functions.  As in the original works by Funk and Radon, we invoke Abel type integrals, which are basic objects of Fractional Calculus
  \cite{Ru96, SKM}. Using  tools of this branch of analysis, we show that the  Funk-Radon-Helgason  method can be successfully applied to derive new inversion formulas, which work well when ``standard''  procedures are inapplicable because of insufficient decay  of functions at infinity or lack of smoothness. Importance of  fractional integration in integral geometry was pointed out in a series of publications;  see, e.g.,  \cite {Gi, Ru98}.

A customary ingredient of the method of mean value operators for the  Radon transform $\vp=Rf$ over $k$-dimensional totally geodesic submanifolds is differentiation of  form
$\partial/\partial r^2$.
 For example, if $k$ is even, then \be f(x)=\pi^{-k/2}
\Big (-\frac{\partial}{\partial r^2} \Big )^{k/2} (R^*_x \vp) (r) \Big
|_{r=0};
\ee
cf. \cite[p. 113]{Ru04b} for the $k$-plane transform on $\bbr^n$ or \cite[p. 128]{He} (with minor changes) for the similar transform on the hyperbolic space. Here $(R^*_x \vp) (r)$ is a certain mean value of $\vp$, which is called the shifted dual Radon transform \cite{Rou} (precise definitions will be given later).
 A remarkable observation by Helgason \cite [p. 116]{He} is that for an even $k$, to reconstruct $f$ from the $k$-plane transform $\vp=R f$, it suffices to differentiate in $r$, rather than in $r^2$. Helgason's result reads as follows:
\be\label {hel}
f(x)=c_k \left [ \Big (\frac{\partial}{\partial r} \Big )^{k} \, (R^*_x \vp)(r)\right ]_{r=0},\qquad c_k=\const, \quad \text{\rm $k$ even}.
\ee
In the present article we show that compositions of the Erd\'elyi-Kober type fractional integrals with power weights
 yield more inversion formulas with   ``usual'' differentiation $(\partial/\partial r)^k$. This is done  for  totally geodesic transforms in {\it all} dimensions on  arbitrary constant curvature space  and under  minimal assumptions for $f$.

Pointwise inversion of  Radon-like transforms  of nonsmooth functions with minimal assumptions at infinity was studied in
\cite{BR,  Ru02a, Ru02b, Ru04a, Ru04b, So, Str}, where one can find further  references. Methods of these papers mainly deal with different kinds of singular integrals, wavelet transforms, and Riesz potentials. To
the best of our knowledge, applicability of the  method of mean value operators (which is historically the first) to such ``rough" functions was not investigated before. The differentiation issue  related to (\ref{hel}) is especially appealing. Some Radon inversion formulas with ``usual" differentiation, but for smooth rapidly decreasing functions, were obtained in \cite {AR, Mad}. The method of these papers differs from ours in principle.

{\bf Plan of the paper.} Section 2 contains necessary
preliminaries from Fractional Calculus. Sections 3,4, and 5 deal with totally geodesic Radon transforms on the Euclidean space  $\bbr^n$, the $n$-dimensional hyperbolic space $\bbh^n$, and  the unit sphere $S^n$, respectively. In Section 6 we give detailed proof of the Helgason's formula
(\ref{hel}), evaluate the constant $c_k$ (that was not done in \cite{He}), discuss related results and open problems.

{\bf Notation and conventions.}  In the following $ \; \sigma_{n-1}  \!= \!2 \pi^{n/2}/
 \Gamma(n/2)$ is the area of the unit sphere $S^{n-1}$ in $\rn; \;
 e_1, \ldots, e_n$ are coordinate unit
 vectors; $o$ is the origin of $\rn$. We say that $f$ is a locally integrable
 function on $\bbr_+=(0,\infty)$  (resp., on $\bbr^n\setminus \{o\}$),  if it is Lebesgue integrable on any interval $(a,b)$, $0<a<b<\infty$ (resp., on any shell
   $0<a<|x|<b<\infty$). The letter $c$ stands for a constant, which can be different at each occurrence. More notation will be introduced in due course.

\section{Preliminaries from Fractional Calculus}

We recall basic  facts
about fractional integration and differentiation  with the main focus on integral-geometric
 applications in subsequent sections. More information can be found in \cite {Ru96, SKM}.

\setcounter{equation}{0}
\subsection  {Riemann-Liouville fractional integrals}
For a sufficiently good function $f$  on  $\bbr_+=
(0,\infty)$, we consider  two types of the Riemann-Liouville\footnote{In many publications  $I^\a_{-}f$   is called the Weyl fractional integral.
However,  Weyl's publication is dated by 1917 and focused on periodic functions, while  Liouville's work, dealing with such integrals, was published in 1832; see historical notes in \cite {SKM}. } fractional integrals of order $\alpha
> 0$:
\[ (I^\a_{+}f ) (t) =
\frac{1}{\Gamma (\alpha)} \intl^t_0 \frac{f(s) \,ds} {(t - s)^{1-
\alpha}},  \quad (I^\a_{-}f ) (t) = \frac{1}{\Gamma (\alpha)} \intl_t^{\infty}
\frac{f(s) \,ds} {(s-t)^{1- \alpha}}. \]
 Existence of $I^\a_{-}f$ essentially depends on the behavior of $f$ at infinity.

 \begin{lemma}\label{lif} Let $f$ be a 
 locally integrable function on $\bbr_+$.
 Then
 $(I^\a_{-}f)(t)$ is finite for almost all $t>0$ provided
  \be\label{for10} \int_1^\infty |f(s)|\, s^{\a -1}\, ds
<\infty.\ee
If $f$ is non-negative and (\ref{for10}) fails, then $(I^\a_{-}f)(t)=\infty$ for every $t\ge 0$.
\end{lemma}
\begin{proof}  Note that the lower limit of integration in (\ref{for10})  can be replaced by  any number $a>0$.
To prove the first statement, it suffices to show that
$$\intl_a^b (I^\a_{-}|f|)(t)\,dt<\infty$$ for any $0<a<b<\infty$. This can be done by
 changing the order of integration and using (\ref{for10}). To
prove the second statement, we assume the contrary, that is,
$(I^\a_{-}f)(t)$ is finite, but (\ref{for10}) fails. Choose any
$a>t$ and $N>0$, and let first $\a\le 1$.  Since $(s-t)^{1-\a}\le
s^{1-\a}$, then \bea \int_t^{a+N}\frac{f(s)\,
ds}{(s-t)^{1-\a}}&>&\int_t^{a+N}\frac{f(s)\,
ds}{s^{1-\a}}>\int_a^{a+N}\frac{f(s)\,
ds}{s^{1-\a}}\nonumber\\&=&\left ( \,\int_1^{a+N}-\int_1^{a}\right
)\frac{f(s)\, ds}{s^{1-\a}}.\nonumber\eea If $N \to \infty$, then,
by the assumption, the left-hand side remains bounded whereas the
right-hand side tends to infinity.  For $\a>1$, we proceed as follows: \bea
\int_t^{2a+N}(s-t)^{\a-1}f(s)\,
dy&>&\int_{2a}^{2a+N}(s-t)^{\a-1}f(s)\,
dy\nonumber\\&>&2^{1-\a}\int_{2a}^{2a+N}s^{\a-1}f(s)\, ds
\nonumber\eea (note that $s-t>s-a>s/2$). The rest of the proof is
as above.
\end{proof}

Changing the order
of integration, we easily get \be\label{sg-} I^\a_{\pm}I^\b_{\pm}f=I^{\a+\b}_{\pm}f, \quad
I^\a_{-} \,t^{-\a-\b }\,I^\b_{-}f=t^{-\b }\,I^{\a+\b}_{-}\, t^{-\a}f,\ee
provided that
integrals in either side exist in the Lebesgue sense.
 Here and on, powers of $t$  stand for the corresponding
 multiplication operators (instead of $t$ there may be another letter).
  The second equality
can be alternatively derived from   $ I^\a_{+}I^\b_{+}f=I^{\a+\b}_{+}f$ if we replace variables  by their reciprocals.

Fractional derivatives $\D^\a_{\pm}\vp$  of order $\a>0$ are defined as left
inverses of the corresponding fractional integrals, so that
\be\label{09zsew}\D^\a_{\pm}I^\a_{\pm} f=f.\ee
 Operators  $\D^\a_{\pm}$ may have different analytic forms,  depending on the class of functions.
For example, if $\alpha = m +
\alpha_0, \;
 m = [\alpha], \; 0 \le \alpha_0 <
1$, then
 \be\label{frr+}\Cal D^\a_{\pm} \vp = (\pm d/dt)^{m +1} I^{1 -
\alpha_0}_{\pm} \vp. \ee
The  existence of the fractional derivative  and the equality (\ref{09zsew}) must be justified at each occurrence.

\subsection  { Modified Erd\'elyi-Kober fractional integrals}\label{drtey}

 There exist many modifications and generalizations of the Riemann-Liouville  integrals, for instance,
\[ (I^\a_{+, 2} f)(t)\!=\!
\frac{2}{\Gamma (\alpha)} \intl^t_0 \!\frac{f(s) \, s\,ds} {(t^2 - s^2)^{1-
\alpha}},  \quad (I^\a_{-, 2} f)(t)\! =\! \frac{2}{\Gamma (\alpha)} \intl_t^{\infty}\!
\frac{f(s) \, s\,ds} {(s^2-t^2)^{1- \alpha}}. \]
We call $I^\a_{\pm, 2} f$ the modified Erd\'elyi-Kober fractional integrals. They differ from the classical Erd\'elyi-Kober  integrals, as in \cite{SKM, Sn}, by weight factors. Clearly,
\be\label{as34b}  I^\a_{\pm, 2} f=A^{-1}I^\a_{\pm} Af, \qquad    (Af) (t)=f (\sqrt {t}).
\ee
The following statement
 is a consequence of Lemma \ref{lif}:
 \begin{lemma}\label{lifa2} Let $f$ be a
 locally integrable function on $\bbr_+$.
 Then
 $(I^\a_{-, 2} f)(t)$ is finite for almost all $t>0$ provided
  \be\label{for10z} \intl_1^\infty |f(s)|\, s^{2\a -1}\, ds
<\infty.\ee
If $f$ is non-negative and (\ref{for10z}) fails, then $(I^\a_{-, 2} f)(t)\!=\!\infty$ for every $t\!\ge \!0$.
\end{lemma}

\begin{lemma} The following formulas hold provided that integrals on the right-hand side exist in the Lebesgue sense:
\be\label{sg-1} I^\a_{\pm, 2} I^\b_{\pm, 2} f =I^{\a+\b}_{\pm,2}f,\ee
\be\label{sgwe1}
I^\a_{-,2} \,t^{-2\a-2\b }\,I^\b_{-,2}f=t^{-2\b }\,I^{\a+\b}_{-,2}\, t^{-2\a}f.\ee
\be\label{sgwe2}
t\,I^\a_{-,2} \,t^{-2\a-1}\,I^\a_{-,2}f=2^{2\a}I_{-}^{2\a}f, \quad I^\a_{+,2} \,t^{1-2\a}\,I^\a_{+,2}f=2^{2\a}I_{+}^{2\a}\, t\,f.\ee
\end{lemma}
\begin{proof}
Owing to (\ref{as34b}),  the first two formulas are immediate consequences of  (\ref{sg-}).
To prove the first formula in (\ref{sgwe2}), we change the order of integration  and get
 \[ l.h.s.=\frac{4}{\Gam^2(\a)}\int_t^\infty f(s)\, I(s,t)\, ds,\]
 where
 \bea &&I(s,t=st\intl_t^s (s^2 -r^2)^{\a -1}(r^2 - t^2)^{\a -1} r^{-2\a}\, dr\nonumber\\
 &&\qquad \text {\rm (set $\eta=(r^2 - t^2)/(s^2 - t^2)$) }\nonumber\\
 &&=\frac{s\, (s^2 - t^2)^{2\a-1}}{2t^{2\a}}\intl_0^1 \eta^{\a-1}  (1-\eta)^{\a-1}
 \left (1-\eta \left (1-\frac{s^2}{t^2} \right )\right )^{-\a-1/2}\, d\eta. \nonumber\eea
The last integral represents the hypergeometric function and can be evaluated using formulas
2.1.3(10) and  2.8.(6) from \cite {Er}. This gives
\[
 I(s,t)=\frac{2^{2\a-2}\,\Gam^2(\a)}{\Gam (2\a)}\, (s-t)^{2\a-1},\]
and the result follows. The second formula in (\ref{sgwe2}) can be obtained from the first
one if we replace variables by their reciprocals.
\end{proof}

\begin{remark}{\rm
   Formulas in (\ref{sgwe2}), which express
 compositions of  Erd\'elyi-Kober type integrals through the usual Riemann-Liouville integrals, are not well-known. They occur in the more general context related to Gegenbauer transformations; cf. \cite [p. 120, formula (4.19)]{Dea},  \cite[Theorem 2.2]{vBE}.
 These formulas  play an important role in our consideration; see Remark \ref{pp00bv}}.
\end{remark}

 Fractional derivatives of the  Erd\'elyi-Kober type   can be defined as the  left inverses
$ \Cal D^\a_{\pm, 2} = (I^\a_{\pm, 2})^{-1}$.
By (\ref{as34b}),
\be\label {0u8n}
\Cal D^\a_{\pm, 2}\vp = A^{-1}\Cal D^\a_{\pm} A \vp, \qquad    (Af) (t)=f (\sqrt {t}),\ee
where the  Riemann-Liouville  derivatives $\Cal D^\a_{\pm}$ can be
chosen in different forms, depending on our needs.
For example, if $\alpha = m +
\alpha_0, \;
 m = [\alpha], \; 0 \le  \alpha_0 <
1$, then, formally, (\ref{frr+}) yields
\be\label{frr+z}   \Cal D^\a_{\pm, 2} \vp=(\pm D)^{m +1}\,
I^{1 - \alpha_0}_{\pm, 2}\vp, \qquad D=\frac {1}{2t}\,\frac {d}{dt}.\ee

Inversion of  $I^\a_{-, 2}$ may cause  difficulties.
 Let, for instance, $\alpha = m +\alpha_0, \; m = [ \alpha], \; 0 <\alpha_0 <1$. Then
 the  standard complementation procedure, as in (\ref{frr+z}), can  be inapplicable. Indeed, this formula assumes convergence of the integral
 $I^{1 - \alpha_0}_{-, 2}\vp=I^{1 - \alpha_0}_{-, 2}I^{m+ \alpha_0}_{-, 2}f=I^{m+1}_{-, 2}f$ or, equivalently,
 $\int_1^\infty f(t)\, t^{2m +1}\, dt
<\infty$.
The latter is not guaranteed by (\ref{for10z}),   however, this  obstacle can be circumvented.

\begin{theorem}\label{78awqe} Let $\vp= I^\a_{-, 2} f$, where $f$ is a locally integrable
 function on $\bbr_+$, satisfying  (\ref{for10z}). Then  $f(t)= (\Cal D^\a_{-, 2} \vp)(t)$ for almost all $t\in \bbr_+$, where  $\Cal D^\a_{-, 2} \vp$ has one of the following forms.

\noindent {\rm (i)} If $\a=m$ is an integer, then

\be\label {90bedri}
\Cal D^\a_{-, 2} \vp=(- D)^m \vp, \qquad D=\frac {1}{2t}\,\frac {d}{dt}.\ee

\noindent {\rm (ii)}  If $\alpha = m +\alpha_0, \; m = [ \alpha], \; 0 \le \alpha_0 <1$, then

\be\label{frr+z3} \Cal D^\a_{-, 2} \vp = t^{2(1-\a+m)}
(- D)^{m +1} t^{2\a}\psi, \quad \psi=I^{1-\a+m}_{-,2} \,t^{-2m-2}\,\vp.\ee
Alternatively,
\be\label{frr+z4} \Cal D^\a_{-, 2} \vp=2^{-2\a}\, \Cal D^{2\a}_- \, t\,  I^\a_{-, 2}\, t^{-2\a-1} \, \vp,\ee
where $D^{2\a}_-$ denotes the Riemann-Liouville derivative of order $2\a$, which can be computed according to (\ref{frr+}).

\noindent {\rm (iii)} If, moreover, $\int_1^\infty |f(t)|\, t^{2m +1}\, dt
<\infty$, then
\be\label{frr+z4y0} \Cal D^\a_{-, 2} \vp=(- D)^{m +1} I^{1-\a+m}_{-, 2}\, \vp.\ee
\end{theorem}
\begin{proof}  {\rm (i)}  is obvious. To prove (\ref{frr+z3}),
we swap $\a$ with $\b$ in (\ref{sgwe1}) to get
\be\label{nvtyu}
I^{\a+\b}_{-,2}\, t^{-2\b}f=t^{2\a }I^\b_{-,2} \,t^{-2\a-2\b }I^\a_{-, 2} f.\ee
 The existence of
$I^{\a+\b}_{-,2}\, t^{-2\b}f$ is guaranteed by  (\ref{for10z}). Choosing $\b=1-\a+m$,  we obtain (\ref{frr+z3}).
To prove (\ref{frr+z4}), we observe that the existence conditions for the  Erd\'elyi-Kober type integral $I^\a_{-, 2} f$  and
 the Riemann-Liouville integral $I^{2\a}_{-} f$  coincide. Hence,
(\ref{sgwe2}) yields the result. Formula (\ref{frr+z4y0}) follows from the semigroup property
$I^{1-\a_0}_{-, 2}I^{m +\alpha_0}_{-, 2}f=I^{m +1}_{-, 2}f $ owing to (i).
\end{proof}

\begin{remark}\label{pp00bv} {\rm An advantage of the inversion formula  (\ref{frr+z4}), which follows from  (\ref{sgwe2}), in comparison with (\ref{90bedri}),
 (\ref {frr+z3}), and (\ref {frr+z4y0}), is that
 it employs the derivative $d/dt$ rather than  $D=(2t)^{-1} d/dt=d/dt^2$.  Similarly, the second formula in (\ref {sgwe2}) yields
 \[ \Cal D^\a_{+, 2} \vp=2^{-2\a} t^{-1} \D_{+}^{2\a}\,I^\a_{+,2} \,t^{1-2\a}\, \vp.\]
 In particular, if  $\a=k/2$, $k\in \bbn$, then
\be\label{nzxceqq} \Cal D^{k/2}_{+, 2} \vp =2^{-k} t^{-1} \left (\frac {d}{dt}\right )^k \,I^{k/2}_{+,2} \,t^{1-k}\, \vp, \ee
\be\label{nzxceq}  \Cal D^{k/2}_{-, 2} \vp =2^{-k}\left (-\frac {d}{dt}\right )^k \, t \,I^{k/2}_{-, 2} \,t^{-k-1}\,\vp.\ee
This  observation will
 be used in inversion formulas for operators of integral geometry in the next sections.}
\end{remark}

\section  {The $k$-plane transforms }

\setcounter{equation}{0}

Let $\cgnk$  be
 manifold of all non-oriented  $k$-planes $\t$ in
$\rn$; $\gnk$ is the  Grassmann manifold
 of  $k$-dimensional  linear subspaces $\zeta$ of $\rn$; $1 \le k\le n-1$.
 Each $k$-plane $\tau$  is parameterized by the pair
$( \zeta, u)$ where $\zeta \in \gnk$ and $ u \in \zeta^\perp$ (the
orthogonal complement of $\zeta $ in $\rn$). Clearly, $\cgnk$ is a bundle over $\gnk$   with an
$(n-k)$-dimensional fiber.
The manifold  $\cgnk$ is endowed with the product measure $d\t=d\zeta du$,
where $d\z$ is the
 $SO(n)$-invariant measure  on $\gnk$ of  total mass
$1$, and $du$ denotes the usual volume element on $\zeta^\perp$; cf. \cite[Chapter 3]{Matt}.

The $k$-plane transform  $Rf$ of a function $f$ on $\bbr^n$  is defined
 by  $$(Rf)(\tau) \equiv (Rf)(\zeta, u) = \intl_\zeta f(u+v)\,
dv$$ provided that this integral is meaningful. According to the general Funk-Radon-Helgason scheme, to reconstruct $f$ from $\vp=Rf$ we need the following mean value operators
\be (\M_x f)(r)=\intl_{SO(n)} f(x+r\gam e_n)\, d\gam,\qquad r>0,
  \ee
\be(R^*_x \vp)(r)= \intl_{SO(n)} \!  \vp (\gam \bbr^k +x + r\gam e_n) \,
 d\gam. \ee
Here $(\M_x f)(r)$ is the usual spherical mean of $f$,    $\bbr^k= \bbr e_1  \oplus \cdots \oplus \bbr e_k$, and $(R^*_x \vp)(r)$ averages $\vp$ over all $k$-planes at distance $r$ from $x$.

\begin{lemma} \label {mlkii} Let $f$ be a locally integrable function on $\bbr^n\setminus \{o\}$.
 If $f$ is radial, i.e., $f(x)\equiv f_0 (|x|)$, then
\be \label {gyuyyr}(Rf)(\tau)=\pi^{k/2} \,(I^{k/2}_{-,2} f_0 )(r), \qquad r=\dist (o, \t).\ee
More generally,
\be\label {gyuyy} (R^*_xRf)(r)=\pi^{k/2} \,(I^{k/2}_{-,2} \M_x f)(r).\ee
These equalities hold  provided that  expressions on either side are finite when $f$ is
replaced by $|f|$.
\end{lemma}
Formulas (\ref{gyuyyr}) and (\ref{gyuyy}) can be found in  \cite [p. 98, 110] {Ru04b} and  \cite[p. 118]{He},
where the notation for the Erd\'elyi-Kober operators is not used. The first formula is a particular case of the second one, corresponding to $x=0$ and $f$ radial.

\begin{theorem} \label{7zvrt} Let  $f$ be a  locally integrable function on $\bbr^n\setminus \{o\}$. If
\be\label{for10zk} \intl_{|x|>1} \frac{|f(x)|}{|x|^{n-k}}\, dx
<\infty,\ee
then $(Rf)(\tau)$ is finite for almost all $\t\!\in \!\cgnk$.
If  $f$ is nonnegative, radial,  and (\ref{for10zk}) fails, then $(Rf)(\tau)\!=\!\infty$ for every $\t\!\in\! \cgnk$.
\end{theorem}
\begin{proof} Let $f_0 (r)=\int_{SO(n)} f(r\gam e_n)\, d\gam$. Owing to  (\ref{for10zk}), we have
 $\int_1^\infty |f_0(r)|\, r^{k -1}\, dr<\infty$ and, therefore, by Lemma \ref{lifa2},
 $(I^{k/2}_{-,2}f_0)(r)$ is finite for almost all $r>0$. On the other hand,
\[\intl_{SO(n)} (Rf) (\gam (\bbr^k + re_n))=(R^*_xRf)(r)\Big |_{x=0}=
\pi^{k/2} \,(I^{k/2}_{-,2}f_0)(r).\]
Hence, $(Rf) (\gam (\bbr^k + re_n))<\infty$ for almost all pairs $(\gam, r) \in SO(n)\times \bbr_+$. It means that
$(Rf)(\tau)$ is finite for almost all $\t\in \cgnk$.
 If, for nonnegative $f\!\equiv \!f_0 (|x|)$, (\ref{for10zk}) fails, then $\int_1^\infty \!f_0(r)\, r^{k -1}\, dr \!\equiv \!\infty$. Hence, Lemma \ref{lifa2} and (\ref{gyuyyr}) yield $\pi^{k/2} \,(I^{k/2}_{-,2} f_0 )(r)\!=\!(Rf)(\tau)\!\equiv \!\infty$.
\end{proof}

According to the general Funk-Radon-Helgason formalism,
our next aim is to reconstruct $(\M_x f)(r)$ from the equality
\be\label{uutrr}(I^{k/2}_{-,2} \M_x f)(r)=\pi^{-k/2} (R^*_xRf)(r)\ee
(cf. (\ref{gyuyy})), and then pass to the limit as  $r\to 0$. We consider two classes of functions $f$
satisfying (\ref{for10zk}). The first one, denoted by $C_\mu (\bbr^n)$, consists of
continuous functions of order $O(|x|^{-\mu})$. If $\mu>k$, then $(Rf)(\tau)$
is finite for every $\t\in \cgnk$.  The second class is $L^p (\bbr^n)$. If
$1\le p<n/k$, then by H\"older's inequality, (\ref{for10zk}) is satisfied, and
therefore, $(Rf)(\tau)$ is finite for almost all $\t\in \cgnk$.\footnote{Different proofs
of this statement can be found in \cite{Ru04b, So, Str}.}
The conditions $\mu>k$ and $1\le p<n/k$ are sharp. It means that there are functions
$f_1 \in C_\mu (\bbr^n)$ and $f_2\in L^p (\bbr^n)$ such that $Rf_1$  and
$Rf_2$ are identically infinite if $\mu\le k$ and $p\ge n/k$, respectively.  For example, in the second case one can take
\be f_2 (x)=\frac{(2+|x|)^{-n/p}}{\log^{1/p+\del} (2+|x|)}, \qquad 0<\del <1/p', \quad 1/p +1/p'=1.\ee
For this function, the integral on the right-hand side of  (\ref{gyuyyr}) diverges.

\begin{lemma}\label{jjyytt} ${}$\hfill

 \noindent {\rm (i)} \ If  $f\in C_\mu (\bbr^n)$, $\mu >k$,  then for every $x\in\bbr^n$ and  $r>0$, the spherical mean  $(\M_x f)(r)$ can be recovered from $\vp=Rf$ by the formula
\be\label{uimuy}
(\M_x f)(r)=\pi^{-k/2} (\Cal D^{k/2}_{-, 2} R^*_x \vp)(r),\ee
where the  Erd\'elyi-Kober differentiation operator
$\Cal D^{k/2}_{-, 2}$ can be computed by (\ref{90bedri}),
 (\ref {frr+z3}), or (\ref{frr+z4}). Under the stronger assumption $\mu >2+2[k/2] \, (>k)$,
 $\Cal D^{k/2}_{-, 2}$ can also be computed by
(\ref {frr+z4y0}).

 \noindent {\rm (ii)}  \ If $f\in L^p (\bbr^n)$, $1\le p<n/k$, then (\ref{uimuy}) holds for every $r>0$ and almost all $x\in\bbr^n$
 with $\Cal D^{k/2}_{-, 2}$  computed by (\ref{90bedri}),
 (\ref {frr+z3}), or (\ref{frr+z4}).
If, moreover, $p<n/(2+2[k/2])\, (<n/k)$,
then  $\Cal D^{k/2}_{-, 2}$ can also be computed by
(\ref {frr+z4y0}).
\end{lemma}

\begin{proof} We need to  justify (a) the validity  of (\ref{uutrr}) for our classes of functions, and (b) applicability of Theorem \ref{78awqe}.

(a) It suffices to  verify   convergence of $(I^{k/2}_{-,2} \M_x f)(r)$
in (\ref{uutrr}) for nonnegative $f$. If $f\in C_\mu (\bbr^n)$, $\mu >k$,  then
\be\label{mmmjjj} (\M_x \, f)(t)\le
c\intl_{S^{n-1}}\frac{d\theta}{|x+t\theta|^\mu}\le c_x \,t^{-\mu},\ee
where $c_x$ is finite. Hence, $(I^{k/2}_{-,2} \M_x f)(r)<\infty$  for every $x$ and $r$. If $f\in L^p (\bbr^n)$, $1\le p<n/k$, then
\bea &&(I^{k/2}_{-,2} \M_x f)(r)=
\frac{2}{\Gam (k/2)\, \sig_{n-1}}\intl_r^\infty
(t^2-r^2)^{k/2 -1}\,t\, dt\intl_{S^{n-1}} f(x-t \theta ) \, d \theta\quad\nonumber \\
&& =\frac{2}{\Gam (k/2)\, \sig_{n-1}}\Bigg (\,\intl_{r<|y|<2r} + \intl_{|y|>2r}\Bigg )
f(x-y)(|y|^2-r^2)^{k/2 -1}\,\frac{dy}{|y|^{n-2}}. \nonumber\eea
 The first
integral is finite for almost all $x$, because it has
a finite  $L^p$-norm (use Minkowski's
inequality for integrals). The second integral does not exceed $$
c\intl_{|y|>2r}f(x-y)\,\frac{dy}{|y|^{n-k}}, \qquad
c=c(t,\a).$$ By H\"older's inequality, it is bounded for all $x$ when $p<n/k$.
Thus, $(I^{k/2}_{-,2} \M_x f)(r)<\infty$  for every   $r>0$ and almost all $x$.

We have proved that (\ref{uutrr})
holds for  every $r>0$ and all or almost all $x$, depending on whether $f\in C_\mu$ or $f\in L^p $.

(b) To obey Theorem \ref{78awqe}, we have to show that $(\M_x f)(t)$ is
locally integrable on $\bbr_+$ and
\be\label{mmzzza}\intl_1^\infty |(\M_x f)(t)|\, t^\lam\, dt<\infty,\ee
where $\lam=k-1$ for formulas (\ref{90bedri}),
 (\ref {frr+z3}),  (\ref{frr+z4}), and $\lam=2[k/2] +1$ for (\ref {frr+z4y0}).
 If $f\in C_\mu (\bbr^n)$,   both statements are immediate consequences of
(\ref{mmmjjj}) and the assumptions for $\mu$. If $f\in L^p (\bbr^n)$, $p\neq 1$, then, for any $0<a<b<\infty$  and   all $x$,
\bea\intl_a^b |(\M_x f)(t)|\, dt &\le& \frac{1}{\sig_{n-1}}\intl_{a<|y|<b}\frac{|f(x-y)|}{|y|^{n-1}}\, dy
\nonumber\\
&\le& \frac{||f||_p}{\sig_{n-1}}\Bigg (\,\intl_{{a<|y|<b}}\frac{dy}{|y|^{(n-1)p'}}\Bigg )^{1/p'}<\infty.\nonumber\eea
Similarly,
\[ \intl_1^\infty |(\M_x f)(t)|\, t^\lam\, dt\le
\frac{||f||_p}{\sig_{n-1}}\Bigg (\,\intl_{|y|>1}\frac{dy}{|y|^{(n-\lam-1)p'}}\Bigg )^{1/p'}<\infty\]
in view of the assumptions for $\lam$  and  $p$. If $p=1$ the changes in this reasoning are obvious.

To complete the proof, we note that Theorem \ref{78awqe} gives us $(\M_x f)(r)$ only for almost all $r$. However, if $f\in C_\mu (\bbr^n)$, then
$(\M_x f)(r)$ is continuous (in both variables), and if $f\in L^p (\bbr^n)$, then  $(\M_x f)(r)$ is an $L^p$-valued continuous function of $r$.
It follows that the inversion formula (\ref{uimuy}) holds  for every $r>0$ for both $C_\mu$ and $L^p$ spaces. However, in the first case it is valid for all $x\in\bbr^n$, and in the second case for almost  all $x$.
\end{proof}

 Lemma  \ref{jjyytt}  and Theorem \ref{78awqe} imply the following theorem,  containing the main inversion results for $Rf$.

 \begin{theorem}\label{invr1abf}
A function $f \in C_\mu  (\bbr^n)$,  $\mu>k$, can be recovered from $\vp=Rf$ by the  formula
\be\label{nnxxzz}f(x) =  \lim\limits_{t\to 0}\, \pi^{-k/2} (\Cal D^{k/2}_{-, 2} R^*_x \vp)(t),\qquad \ee
where the  limit  is uniform on $\bbr^n$ and the  Erd\'elyi-Kober differential operator
$\Cal D^{k/2}_{-, 2}$ can be computed as follows.

\noindent {\rm (i)} If $k$ is even, then
\be\label {90bedrik}
\Cal D^{k/2}_{-, 2} F=(- D)^{k/2} F, \qquad D=\frac {1}{2t}\,\frac {d}{dt}.\ee

\noindent {\rm (ii)}  For any $1\le k\le n-1$,
\be\label{frr+z3k} \Cal D^{k/2}_{-, 2} F = t^{2-k+2m}
(- D)^{m +1} t^{k}\psi, \quad \psi=I^{1-k/2+m}_{-,2} \,t^{-2m-2}\,F,\ee
where $m=[k/2]$. Alternatively,
\be\label{frr+z4k}  \Cal D^{k/2}_{-, 2} F = 2^{-k}  \left (-\frac{d}{dt}\right)^k\, t\,  I^{k/2}_{-, 2}\, t^{-k-1} \,F.\ee
Under the stronger assumption $\mu >2+2[k/2] \, (>k)$,
 $\Cal D^{k/2}_{-, 2}$ can also be  computed as
\be\label{frr+z4yz}\Cal D^{k/2}_{-, 2} F =(- D)^{m +1} I^{1-\a+m}_{-, 2}\,F.\ee
\end{theorem}

Note that (\ref{frr+z4k})
 employs usual differentiation $d/dt$ rather than $D$.

The next theorem contains similar results for $L^p$-functions.

 \begin{theorem}\label{invr1p}
A function $f \in L^p (\bbr^n)$,  $1\le p<n/k$, can be recovered from $\vp=Rf$  at almost every  $x\in \rn$ by the  formula
\be\label{nnxxzz}f(x) =  \lim\limits_{t\to 0}\, \pi^{-k/2} (\Cal D^{k/2}_{-, 2} R^*_x \vp)(t),\ee
where the  limit  is understood in the $L^p$-norm.
Here $\Cal D^{k/2}_{-, 2}$ is  computed as in Theorem \ref{invr1abf}, where  (\ref{frr+z4yz})
is applicable under the stronger assumption $1\le p<n/(2+2[k/2])$.
\end{theorem}

\section{The Hyperbolic Space}
 Let $E^{n, 1}, \, n \ge 2$, be the pseudo-Euclidean space of points $x = (x_1, \dots, x_{n+1}) \in \bbr^{n+1}$ with the inner product
\be \label {ijji}
[x,y]=-x_1 y_1 - \ldots - x_ny_n + x_{n+1} y_{n+1}.\ee
We realize the $n$-dimensional hyperbolic space $X$ as the upper sheet of the two-sheeted hyperboloid
$$
\bbh^n = \{ x \in \bbe^{n, 1}: [x, x] = 1, x_{n+1} > 0\}.
$$
Let $\Xi$ be the set of all $k$-dimensional totally geodesic submanifolds $\xi$ of $X$, $1 \le k\le n-1$.
 As usual, $e_1,\dots,e_{n+1}$ denote  the coordinate unit vectors. We set  $\bbr^{n+1}= \bbr^{n-k}\oplus \bbr^{k+1}$, where
 $$
\bbr^{n-k}=  \bbr e_1 \oplus \ldots \oplus \bbr e_{n-k}, \qquad \bbr^{k+1}=\bbr e_{n-k+1} \oplus \ldots \oplus  \bbr e_{n+1},$$
 and identify $\bbr^{k+1}$ with the pseudo-Euclidean space  $E^{k, 1}$.

  In the
following $x_0 =(0,\dots,0,1)$ and $\xi_0 = \bbh^n\cap\bbe^{k,1}= \bbh^k$ denote the origins in $X$ and $\Xi$ respectively; $G=SO_0(n,1)$ is the identity component of the pseudo-orthogonal group $O(n,1)$ preserving the bilinear form (\ref{ijji}); $K=SO(n)$ and $H=SO(n-k) \times SO_0(k, 1)$ are the isotropy subgroups of $x_0$ and $\xi_0$, so that $X = G/K, \; \Xi=G/H$. One can write $f(x) \equiv f(gK), \; \varphi(\xi)\equiv\varphi(gH),\; g\in G$.  The geodesic distance between points $x$ and $y$ in $X$ is defined by $d(x,y) = \cosh^{-1}[x,y]$.
 Each $x \in X$ can be represented in the
hyperbolic polar coordinates as \be x =
\zeta\sinh\om +e_{n+1}\cosh\om,\ee
where $\zeta$ is a point of the unit sphere $S^{n-1}$ in the plane $x_{n+1}=0$ and $0\le \om<\infty$.
In these coordinates the Riemannian measure $dx$ on $X$ has the form $dx=\sinh^{n-1}\om  \, d\om d\sig(\zeta)$, so that
\be\label{aasszx} \intl_X f(x)\, dx\!=\!\intl_0^\infty \sinh^{n-1}\om  \, d\om \intl_{S^{n-1}}\!\! f(\z\sinh \, \om \!+\!e_{n+1}\cosh\, \om)\,d\sig(\zeta).\ee
In particular, if $f$ is $K$-invariant (or zonal), that is, $f(x)\equiv f_0 (\cosh\omega)=f_0 (x_{n+1})$, then
\be\label{ppooii}
\intl_X f(x)\, dx=\sig_{n-1}\intl_1^\infty f_0(s) (s^2 -1)^{n/2 -1}\, ds.\ee
The space $L^p(X)$ (with respect to $dx$ above) is defined in a standard way; $C(X)$ is the
space of continuous functions on $X$; $C_0 (X)$ denotes the space of continuous functions
on $X$ vanishing at infinity.
We also define
\[C_\mu (X)=\{f\in C(X): f(x)= O(x_{n+1}^{-\mu})\}.\]
Of course, it might be natural to compare functions at infinity with powers of the geodesic distance $d(x_0,x)$ (as in the case of $\bbr^n$), however, it is technically more convenient to use powers of $x_{n+1}=\cosh d(x_0,x)$.

For $x \in X$ and $\xi \in \Xi$, we denote by $r_x$ and $r_\xi (\in G)$ arbitrary hyperbolic rotations satisfying $r_x x_0 = x, \; r_\xi \xi_0 = \xi$, and write $$f_\xi (x) = f(r_\xi x), \qquad \vp_x(\xi) = \vp(r_x\xi).$$
The  totally geodesic Radon transform $(R f)(\xi)$ of a sufficiently good function $f$ on $X=\bbh^n$
 is defined by
\be\label {004466}
  (R f)(\xi)=\intl_{\xi} f(x) \, d_\xi x\equiv \intl_{SO_0(k,1)} f_\xi(\gamma x_0)d\gamma, \qquad \xi \in \Xi.\ee

The first question is for which functions the integral (\ref{004466}) exists.

 \begin{theorem} \label{llkkjj} ${}$\hfill

\noindent {\rm (i)} If $f \in L^p(X)$, $1 \le p < (n-1)/(k-1)$, then $(Rf)(\xi)$ is finite for almost
all $\xi \in \Xi$.

\noindent {\rm (ii)} If $f \in C_\mu(X)$, $\mu>k-1$, then $(Rf)(\xi)$ is finite for
all $\xi \in \Xi$.

\end{theorem}
\begin{proof}  {\rm (i)} \  We make use of the equality
\be \label{hhyydd}\intl_\Xi \frac{(Rf)(\xi)}{\cosh^n
d(x_0,\xi)}\;d\xi=\intl_Xf(x) \, \frac{dx}{x_{n+1}^{n-k}},\ee
which holds provided that either of these  integrals is finite when $f$ is replaced
by  $|f|$; see \cite[p. 48]{BR}. By H\"older's inequality the right-hand side does not exceed $c\,||f||p$, where, by  (\ref{ppooii}),
\[c^{p'}=\intl_X \frac{dx}{x_{n+1}^{(n-k)p'}}=\sig_{n-1}\intl_1^\infty \frac{(s^2 -1)^{n/2 -1}}{s^{(n-k)p'}}\, ds.\]
The latter is finite if  $1 \le p < (n-1)/(k-1)$.

{\rm (ii)} Consider an arbitrary  $k$-geodesic $\xi\in \Xi$. Let $x_\xi$ be the  point in $\xi$ at the minimum distance
from the origin $x_0$ and let $x$ be an arbitrary point in $\xi$.  Then  the angle
$x_0 x_\xi x$ is $90^\circ$ and, by the hyperbolic trigonometry, see, e.g., \cite[p. 102]{Rat},
\be\label{mmbbvv} x_{n+1}=\cosh d(x_0, x)= \cosh d(x_0, x_\xi)\, \cosh d(x_\xi, x).\ee
 If $f \in C_\mu(X)$, then, denoting $\lam_\xi= \cosh\, d(x_0, x_\xi)$ and using (\ref{mmbbvv}), we obtain
\be\label{mmbbvt} |(Rf)(\xi)|\le c\intl_\xi \frac{d_\xi x}{x_{n+1}^\mu}=c\, \lam_\xi^{-\mu} \intl_\xi \frac{d_\xi x}{\cosh^\mu d(x_\xi, x)}.\ee
 Let $\gam \in G$ be a hyperbolic rotation that sends $\xi_0$ to $\xi$ and such that $\gam x_0=x_\xi$. Changing variable $x=\gam y$,
 and using (\ref{ppooii}),  we write the last integral in (\ref{mmbbvt}) as
 \[\intl_{\xi_0} y_{n+1}^{-\mu}\, dy=\sig_{k-1}\intl_1^\infty (s^2 -1)^{k/2 -1} s^{-\mu}\, ds.\]
This integral is finite if $\mu>k-1$.
\end{proof}
 \begin{remark} The restrictions $\mu>k-1$ and $1 \le p < (n-1)/(k-1)$  are sharp; see Example \ref{wwssa} below.
 \end{remark}

Now we introduce the mean value operators, which will be used in the inversion procedure.
 Let
$$
g_\theta\,=\,\left[\begin{array} {ccc} \cosh\theta &0 &\sinh\theta\\ 0 &I_{n-1} &0\\ \sinh\theta &0 &\cosh\theta \end{array} \right],
$$
 where $I_{n-1}$ is the unit matrix of dimension $n-1$. Clearly,
 \[g_\theta x_0=g_\theta e_{n+1}=
 e_1\sinh \theta + e_{n+1}\cosh \theta. \]
The shifted dual
Radon transform of a function $\vp$ on $\Xi$, which averages $\vp$ over all $\xi \in\Xi$ at geodesic distance $\theta$ from $x$,
 is defined by
\be\label {zss}(R^*_x \vp)(r)=\intl_K\vp_x(\gam g^{-1}_\theta\xi_0)\,d\gam,  \qquad  \sinh \theta= r.
\ee The case $r=0$ corresponds to the usual dual
Radon transform.

  Given $x \in \bbh^n$ and
$ s > 1$, let
 \be\label {2.21h}
 (M_x f)(s) = {(s^2-1)^{(1-n)/2}\over \sigma_{n-1}}
 \intl_{\{y \in \bbh^n:\; [x, y]=s\}} f(y) d\sigma  (y)\ee
 be the spherical mean of
$f$ on $X=\bbh^n$, where $ \; d\sigma  (y)$
 stands for the relevant induced Lebesgue measure.

\begin{lemma}\label {7755vv}  \cite [pp. 131-133]{Liz}, \cite[Lemma 2.1]{BR99}.

\vskip 0.3truecm

\noindent {\rm (i)} $\qquad \sup\limits_{s > 1} \| (M_{(\cdot)} f)(s) \|_p \le
\| f \|_p, \qquad f \in L^p (X), \quad  1 \le p \le \infty$.

\vskip 0.3truecm

\noindent {\rm (ii)}  If $f \in L^p (X), \ 1 \le p < \infty$, then
$\lim\limits_{s \to 1} \| (M_{(\cdot)} f)(s) - f \|_p = 0$.  If $f \in C_0 (X)$, then
$(M_x f)(s)\to f(x)$
as $s\to 1$, uniformly on $X$.
 \end{lemma}

For our purposes it is convenient to set
\be\label {bbx3}
(\tilde M_x f)(t)=(1+t^2)^{-1/2}(M_x f)(\sqrt {1+t^2}).
\ee
Then $\lim\limits_{t \to 0}\,(\tilde M_t f)(x)=f(x)$ as in Lemma \ref{7755vv}.

\begin{lemma} \label {mlkii} Let $f$ be a locally integrable function on $\bbh^n\setminus \{x_0 \}$,
 \[
 \omega=d(x_0 ,x), \qquad  \theta=d(x_0 , \xi).\]
 If $f$  is  zonal, $f(x)\equiv f_0 (\cosh\omega)=f_0 (x_{n+1})$, and
 $$\tilde f_0  (t)=(1+t^2)^{-1/2}f_0 (\sqrt {1+t^2}),$$
 then
  \bea\label {eeqq1h}\qquad
(Rf)(\xi)&=&\frac{\sig_{k-1}}{(1+r^2)^{(k-1)/2}}\intl_r^\infty \tilde f_0  (t)  (t^2-r^2)^{k/2 -1} \,
 t\,dt\nonumber\\
\label {eeqq1hn} &=& \frac{\pi^{k/2}}{(1+r^2)^{(k-1)/2}}\, (I^{k/2}_{-,2} \tilde f_0 )(r),\quad r=\sinh \theta.
 \eea
More generally,
\be\label {gyuyyh} (R^*_x \vp)(r)=\frac{\pi^{k/2}}{(1+r^2)^{(k-1)/2}}\, (I^{k/2}_{-,2} \tilde M_x f)(r).\ee
These equalities hold  provided that  expressions on either side are finite when $f$ is
replaced by $|f|$.
\end{lemma}
 \begin{proof}
It is known that for $ \tau=\cosh  d(x_0 , \xi)$,
 \be
(Rf)(\xi)=\frac{\sigma_{k-1}}{\tau^{k-1}}\int^\infty_{\tau} f_0(s)(s^2-\tau^2)^{k/2-1}ds;
\ee
see \cite[p. 121] {He}, \cite [Lemma 3]{{BR}}.
 Changing variables
$$r=\sqrt {\tau^2-1}, \qquad s=\sqrt {1+t^2},$$
we obtain (\ref{eeqq1h}).
Similarly,
 \[(R^*_x Rf) (\sqrt {\tau^2-1})=\frac{\sig_{k-1}}{\tau^{k-1}}\intl_\tau^\infty (M_x f)(s) (s^2-\tau^2)^{k/2 -1} \,
 ds;\]
cf. \cite [p. 128]{He} and \cite [formula (2.3)]{Ru02a}. This gives (\ref{gyuyyh}).
 \end{proof}

The following statement provides additional information  about the existence of the Radon transform.

\begin{theorem} \label{7zvrt} Let $f$ be a locally integrable function on $\bbh^n\setminus \{x_0\}$ satisfying
\be\label{for10zkh} \intl_{x_{n+1}>2} \frac{|f(x)|}{x_{n+1}^{n-k}}\,\, dx
<\infty.\ee
Then $(Rf)(\xi)$ is finite for almost all $\xi \in\Xi$ .
If, moreover, $f$ is nonnegative, zonal,  and (\ref{for10zkh}) fails, then $(Rf)(\xi)=\infty$ for every $\xi \in\Xi$.
\end{theorem}
\begin{proof} Let $$f_0 (s)\!=\!\intl_{K} \!f(\gam ( e_1\,\sqrt {s^2-1} + e_{n+1}\,s))\, d\gam,  \quad \tilde f_0  (t)\!=\!(1+t^2)^{-1/2}f_0 (\sqrt {1+t^2}).$$
Then (\ref{for10zkh}) is equivalent to
\[I\equiv \intl_a^\infty | \tilde f_0  (t)| t^{k-1}\, dt<\infty, \quad \forall a>0.\]
Indeed, setting $s=\sqrt {1+t^2}$, we get

\bea I&=&\intl_b^\infty |f_0(s)| \,(s^2 -1)^{k/2 -1}\, ds   \qquad  (b=\sqrt {1+a^2}>1)\nonumber\\
&\le& c_b\intl_b^\infty |f_0(s)| \,(s^2 -1)^{n/2 -1}\, \frac{ds}{s^{n-k}} \nonumber\\
&=&c_b\intl_{\cosh^{-1}b}^\infty \frac{\sinh^{n-1} \om }{\cosh^{n-k} \om}\,d\om \intl_{K} \!|f(\gam ( e_1\,\sinh\, \om  + e_{n+1}\,\cosh\, \om))|\, d\gam\nonumber\\
&=&\frac{c_b}{\sig_{n-1}}\intl_{x_{n+1}>b}\frac{|f(x)|}{x_{n+1}^{n-k}}\, dx.\nonumber\eea
 Calculations in the opposite direction are similar. Thus,
 by Lemma \ref{lifa2}, if  (\ref{for10zkh}) holds, then $(I^{k/2}_{-,2} \tilde f_0 )(r)$
 is finite for almost all $r>0$.
  On the other hand, for $r=\sinh \theta$, Lemma \ref{mlkii} yields
\bea&&\intl_{K} (Rf) (\gam g_\theta^{-1}\xi_0)\, d\gam \equiv (R^*_{x_0} Rf)(r)\nonumber\\
&&=\frac{\pi^{k/2}}{(1+r^2)^{(k-1)/2}}\, (I^{k/2}_{-,2} \tilde M_{x_0} f)(r)=
\frac{\pi^{k/2}}{(1+r^2)^{(k-1)/2}}\, (I^{k/2}_{-,2} \tilde f_0)(r).\nonumber\eea
Since $(I^{k/2}_{-,2} \tilde f_0 )(r)$ is finite for almost all $r>0$, then $(Rf) (\gam g_\theta^{-1}\xi_0)<\infty$
for almost all pairs $(\gam, \theta) \in K\times \bbr_+$. It means that
$(Rf)(\xi)$ is finite for almost all $\xi\in \Xi$.

 If, for $f\!\equiv \!f_0 (x_{n+1})\ge 0$, (\ref{for10zkh}) fails, then $\int_1^\infty \!\tilde f_0(t)\, t^{k -1}\, dt \!\equiv \!\infty$. Hence, by Lemma \ref{lifa2}, $(I^{k/2}_{-,2} \tilde f_0 )(r)\!\equiv \!\infty$ and, by (\ref{eeqq1hn}), $(Rf)(\xi)\!\equiv \!\infty$.
\end{proof}
The following example of application  of Theorem \ref{7zvrt} shows that the restrictions  $1 \le p < (n-1)/(k-1)$ and $\mu>k-1$ in Theorem \ref{llkkjj} are sharp.

\begin{example}\label{wwssa} {\rm Consider the following functions
\be f_1 (x)=\frac{x_{n+1}^{(1-n)/p}}{\log (1+x_{n+1})}, \qquad f_2 (x)=\frac{x_{n+1}^{-\mu}}{\log (1+x_{n+1})}.\ee
By  (\ref{ppooii}),
\[ ||f_1||_p^p=\sig_{n-1}\intl_1^\infty  \frac{(s^2 -1)^{n/2 -1}}{s^{n-1}\,\log (1+s)}\, ds <\infty.\]
However, if $p\ge (n-1)/(k-1)$, then (\ref{for10zkh}) fails  because
\[\intl_{x_{n+1}>2}\frac{f_1(x)}{x_{n+1}^{n-k}}\, dx=
\sig_{n-1}\intl_2^\infty \frac{(s^2 -1)^{n/2 -1}}{s^{(n-1)/p}\,\log (1+s)}\, ds =\infty.\]
Similarly,  $f_2\in C_\mu (\bbh^n)$, however, if $\mu\le k-1$, then
\[\intl_{x_{n+1}>2} \frac{f_2(x)}{x_{n+1}^{n-k}}\, dx=
\sig_{n-1}\intl_2^\infty \frac{(s^2 -1)^{n/2 -1}}{s^{n-k+\mu}\,\log (1+s)}\, ds =\infty.\]
}
\end{example}

As in the preceding section, our
 next aim is to reconstruct $(\tilde M_x f)(r)$ from the equality

\be\label{uutrrh}(I^{k/2}_{-,2} \tilde M_x f)(r)=\pi^{-k/2} (1+r^2)^{(1-k)/2}(R^*_{x} Rf)(r)\ee
 (cf. (\ref{uutrr})) and then pass to the limit as  $r\to 0$.

\begin{lemma}\label{jjyyth} Let $X=\bbh^n$. \hfill

 \noindent {\rm (i)} \ If  $f\in C_\mu X)$, $\mu >k-1$,  then for every $x\in X$ and  $r>0$, the spherical mean  $(\tilde M_x f)(r)$ can be recovered from $\vp=Rf$ by the formula
\be\label{uimuyh}
(\tilde M_x f)(r)=\pi^{-k/2} (\Cal D^{k/2}_{-, 2} (1+r^2)^{(1-k)/2} R^*_x \vp)(r),\ee
where the  Erd\'elyi-Kober differentiation operator
$\Cal D^{k/2}_{-, 2}$ is computed by (\ref{90bedri}),
 (\ref {frr+z3}), or (\ref{frr+z4}). Under the stronger assumption $\mu>2[k/2] +1$,
 $\Cal D^{k/2}_{-, 2}$ can also be  computed by
(\ref {frr+z4y0}).

 \noindent {\rm (ii)}  \ If $f\in L^p (X)$, $1\le p<(n-1)/(k-1)$, then (\ref{uimuyh}) holds for every $r>0$ and almost all $x\in X$
 with $\Cal D^{k/2}_{-, 2}$  computed by (\ref{90bedri}),
 (\ref {frr+z3}), or (\ref{frr+z4}).
If, moreover, $p<(n-1)/(2[k/2] +1)$,
then  $\Cal D^{k/2}_{-, 2}$ can also be  computed by
(\ref {frr+z4y0}).
\end{lemma}

\begin{proof} As in Lemma \ref{jjyytt}, we first show   (a)  convergence of $(I^{k/2}_{-,2} \tilde M_x f)(r)$
for nonnegative $f$, and (b)  applicability of Theorem \ref{78awqe}.

(a) We have
\bea I&\equiv&(I^{k/2}_{-,2} \tilde M_x f)(r)\le c\intl_r^\infty (\tilde M_x f)  (t)  (t^2-r^2)^{k/2 -1} \, t\,dt \nonumber\\
&=& c\intl_\rho^\infty (M_x f)(s)  (s^2-\rho^2)^{k/2 -1} \, ds,  \qquad \rho=\sqrt{1+r^2}.\nonumber\eea
Setting $s=\cosh \om$, $\rho=\cosh \theta$, and denoting $f_x (y)=f(\gam_x y)$, where $\gam_x\in G$,  $\gam_x x_0=x$, we continue:
\bea I&\le& c\intl_\theta^\infty (\cosh^2 \om - \cosh^2 \theta)^{k/2 -1}\, \sinh^{2-n}\om \,d\om\intl_{[x,y]=\cosh \om} f(y) \, d\sig (y)\nonumber\\
&=& c\intl_\theta^\infty (\cosh^2 \om - \cosh^2 \theta)^{k/2 -1}\,\sinh \om\, d\om\nonumber\\
&\times&\intl_{S^{n-1}} f_x(\zeta \sinh \om +e_{n+1}\cosh \om)\,  d\sig (\zeta).\nonumber\eea
By (\ref{aasszx})  this gives
\be\label{bbccxz} I\le \intl_{y_{n+1}>\rho}\frac{(y_{n+1}^2-\rho^2)^{k/2 -1}}{(y_{n+1}^2-1)^{n/2 -1}}\,  f_x (y)\, dy.\ee
This integral has a structure of the hyperbolic convolution
\[(Kf)(x)=\intl_{X} f(y) k([x,y])\, dy=\intl_{X} f_x(y) k(y_{n+1})\, dy, \]
which  can be ``lifted" to a convolution operator on $G$.
By Young's inequality \cite[Chapter 5,
 Theorem 20.18]{HR},
\be \label{aaqqwwz} \| Kf\|_q\le \| f\|_p \|k\|_r, \ee
 where $1 \le p \le q \le \infty, \quad 1 - p^{-1}+q^{-1}=r^{-1}$,
\[
 \| k\|^r_r=\sigma_{n-1}\int^\infty_1|k(t)|^r\; \; (t^2-1)^{n/2-1}dt.
\]
We split the integral in (\ref{bbccxz}) in two pieces $I=I_1+I_2$, corresponding to $\rho<y_{n+1}< 2\rho$ and $y_{n+1}> 2\rho$, and
 consider the cases  $f\in C_\mu (X)$ and $f\in L^p (X)$ separately.

 In the first case,  when $f(y)\le c \,y_{n+1}^{-\mu}$, we have
\[ I_1\le c \intl_{\rho<y_{n+1}< 2\rho}\frac{(y_{n+1}^2-\rho^2)^{k/2 -1}}{(y_{n+1}^2-1)^{n/2 -1}}\, \frac{dy}{[\gam_x y, e_{n+1}]^\mu}.\]
This integral can be estimated using (\ref{aaqqwwz}) with $r=1$, $p=q=\infty$, 
as follows: $I_1\le cA$, where
\bea \label{kkzzer}A&=& \intl_{\rho<y_{n+1}< 2\rho}\frac{(y_{n+1}^2-\rho^2)^{k/2 -1}}{(y_{n+1}^2-1)^{n/2 -1}}\, dy\\
&=&c\, \sigma_{n-1}\intl_\rho^{2\rho} (s^2-\rho^2)^{k/2 -1}\,ds <\infty \quad \forall \rho\ge 1\nonumber\eea
(the last equality holds by  (\ref{ppooii})).
 To estimate $I_2$, we apply (\ref{aaqqwwz}) with $r=p'$ and $q=\infty$, to get  $I_2\le c\, B^{1/p}C^{1/p'}$,
\be \label{kkzzer1} B=\intl_{X} y_{n+1}^{-\mu p}\, dy, \qquad C=\intl_{y_{n+1}>2\rho} \frac{(y_{n+1}^2-\rho^2)^{(k/2 -1)p'}}{(y_{n+1}^2-1)^{(n/2 -1)p'}}\,dy.\ee
  The first integral is finite if $p>(n-1)/\mu$, whereas the second one is finite (for every $\rho \ge 1$) if $p< (n-1)/(k-1)$.  Since both conditions are consistent
  when  $\mu >k-1$,  we are done.

Consider the case  $f\in L^p (X)$ and estimate $I_1$ by using (\ref{aaqqwwz}) with $r=1$, $p=q$. This gives
$||I_1||_p\le A||f||_p$, where $A$ has the form (\ref{kkzzer}). For $I_2$, we apply (\ref{aaqqwwz}) with $r=p'$ and $q=\infty$, so that
$I_2\le c\, C^{1/p'}||f||_p$, where $C$ is the same as in (\ref{kkzzer1}). It follows that, for $1\le p<  (n-1)/(k-1)$, $I$ is finite for almost all $x$ and all $\rho \ge 1$.

(b) To justify applicability of Theorem \ref{78awqe}, we have to show that $(\tilde M_x f)(t)$ is
locally integrable on $\bbr_+$ and
\be\label{mmzzzah}\psi (x)\equiv \intl_1^\infty |(\tilde M_x f)(t)|\, t^\lam\, dt<\infty,\ee
where $\lam=k-1$ for formulas (\ref{90bedri}),
 (\ref {frr+z3}),  (\ref{frr+z4}), and $\lam=2[k/2] +1$ for (\ref {frr+z4y0}).
 Proceeding as in Part (a) of the proof of Lemma \ref{jjyyth}, for any $0<a <b<\infty$ we obtain
\[ \intl_a^b |(\tilde M_x f)(t)|\,dt\le c \intl_{a_1<y_{n+1}< b_1} |f_x(y)|\, dy,\]
 for some $1<  a_1< b_1<\infty$ depending on $a$ and $b$.

If $f\in C_\mu (X)$,   $\mu>0$, the last integral is bounded uniformly in $x$ because
$|f_x(y)|\le c\, [\gam_x y, e_{n+1}]^{-\mu}$ and $[\gam_x y, e_{n+1}]\ge 1$. If $f\in L^p (X)$, then, by  (\ref{aaqqwwz}) (with $r=1$, $q=p$),
 the $L^p$-norm of this integral does not exceed $c \,||f||_p$. Hence $(\tilde M_x f)(t)$ is
locally integrable on $\bbr_+$ for almost all $x$.

Similarly we get
\[ \psi (x)\equiv\intl_1^\infty |(\tilde M_x f)(t)|\, t^\lam\, dt\le c\intl_{y_{n+1}>\sqrt {2}}(y_{n+1}^2-1)^{(\lam +1-n)/2}  |f_x(y)|\, dy.\]
If $f\in L^p (X)$, then, as above, $\psi (x)\le c\,c_\lam^{1/p'}\, ||f||_p$, where, for $p>1$,
\be\label {lloohh} c_\lam =\intl_{y_{n+1}>\sqrt {2}}(y_{n+1}^2-1)^{(\lam +1-n)p'/2} \, dy=
\sigma_{n-1}\int^\infty_{\sqrt {2}}(s^2-1)^{\del/2}\, ds,\ee
$\del=(\lam +1-n)p'+n-2$. The last integral is finite provided that
$$
p< (n-1)/\lam=\left \{ \begin{array} {ll} (n-1)/(k-1), & \mbox {\rm if $\lam=k-1$,}\\
(n-1)/(2[k/2] +1), & \mbox {\rm if $\lam=2[k/2] +1$.}\\
\end{array}\right.
$$
If $p=1$, then $\psi (x)\le c\ ||f||_1$ if $\lam \le n-1$. In the case $\lam=k-1$ this inequality does not restrict the range of $k$. If
 $\lam=2[k/2] +1$, then $\lam \le n-1$ is equivalent to $[k/2]\le n/2 -1$ and the case $k=n-1$ with $n$  odd must be excluded.

If $f\in C_\mu (X)$, then
\[ \psi (x)\le c_1\intl_{y_{n+1}>\sqrt {2}}(y_{n+1}^2-1)^{(\lam +1-n)/2}   \frac{dy}{[\gam_x y, e_{n+1}]^\mu}\le c_1\, B^{1/p}c_\lam^{1/p'},\]
where $B$ is the integral from (\ref{kkzzer1}), which is finite if $p>(n-1)/\mu$, and  $c_\lam$ is known from (\ref{lloohh}). It is finite if
 $p< (n-1)/\lam$. Thus, we have to choose $(n-1)/\mu <p<(n-1)/\lam$, which is possible if $\mu>\lam$. If $\lam=k-1$  we arrive at the ``standard'' assumption $\mu >k-1$. If  $\lam=2[k/2] +1$, we get the new restriction $\mu>2[k/2] +1$, as stated in the lemma.

To complete the proof, we note that Theorem \ref{78awqe} gives us $(\tilde M_x f)(r)$ only for almost all $r$. However, if $f\in C_\mu (X)$, then $(\tilde M_x f)(r)$ is continuous (in both variables), and if $f\in L^p (X)$, then  $(\tilde M_x f)(r)$ is an $L^p$-valued continuous function of $r$.
It follows that the inversion formula (\ref{uimuyh}) holds  for every $r>0$ for both $C_\mu$ and $L^p$ spaces. However, in the first case it is valid for all $x\in X$ and in the second case for almost  all $x$.
\end{proof}

Lemma  \ref{jjyyth}  and Theorem \ref{78awqe} yield the main  results for $X=\bbh^n$, which mimic those for $X=\bbr^n$.

 \begin{theorem}\label{invr1abh}
A function $f \in C_\mu  (X)$,  $\mu>k-1$, can be recovered from $\vp=Rf$ by the  formula
\be\label{nnxxzzh}f(x) =  \lim\limits_{t\to 0}\, \pi^{-k/2} (\Cal D^{k/2}_{-, 2} R^*_x \vp)(t),\qquad \ee
in which the  limit  is uniform on $X$ and the  Erd\'elyi-Kober differential operator
$\Cal D^{k/2}_{-, 2}$ can be computed as follows.

\noindent {\rm (i)} If $k$ is even, then
\be\label {90bedrikh}
\Cal D^{k/2}_{-, 2} F=(- D)^{k/2} F, \qquad D=\frac {1}{2t}\,\frac {d}{dt}.\ee

\noindent {\rm (ii)}  For any $1\le k\le n-1$,
\be\label{frr+z3kh} \Cal D^{k/2}_{-, 2} F = t^{2-k+2m}
(- D)^{m +1} t^{k}\psi, \quad \psi=I^{1-k/2+m}_{-,2} \,t^{-2m-2}\,F,\ee
where $m=[k/2]$. Alternatively,
\be\label{frr+z4kh}  \Cal D^{k/2}_{-, 2} F = 2^{-k}  \left (-\frac{d}{dt}\right)^k\, t\,  I^{k/2}_{-, 2}\, t^{-k-1} \,F.\ee
Under the stronger assumption $\mu>2[k/2] +1$,
 $\Cal D^{k/2}_{-, 2}$ can be also computed as
\be\label{frr+z4yh}\Cal D^{k/2}_{-, 2} F =(- D)^{m +1} I^{1-\a+m}_{-, 2}\,F.\ee
\end{theorem}

For $L^p$-functions we have the following.

 \begin{theorem}\label{invr1ph}
A function $f \in L^p (X)$,  $1\le p<(n-1)/(k-1)$, can be recovered from $\vp=Rf$  at almost every point $x$ by the  formula
\be\label{nnxxzzh}f(x) =  \lim\limits_{t\to 0}\, \pi^{-k/2} (\Cal D^{k/2}_{-, 2} R^*_x \vp)(t),\ee
where the  limit  is understood in the $L^p$-norm.
Here $\Cal D^{k/2}_{-, 2}$ is  computed as in Theorem \ref{invr1abh}, where formula (\ref{frr+z4yh})
is applicable under the stronger condition $1\le p<(n-1)/(2[k/2] +1)$.
\end{theorem}

\section{The case $X=S^n$}

We recall some basic facts \cite{He, Ru02b}. Let $\;\bbr^{n+1}=\bbr^{k+1} \times \bbr^{n-k}$,
$$
\bbr^{k+1}=\bbr e_{1} \oplus \ldots \oplus  \bbr e_{k+1}, \qquad \bbr^{n-k}=  \bbr e_{k+2} \oplus \ldots \oplus \bbr e_{n+1};$$
 $\sig_n=|S^n|=2\pi^{(n+1)/2}/\Gam ((n+1)/2); \;
 \xi_0=S^k$ is the unit sphere in $\bbr^{k+1}$;  $d (\cdot, \cdot)$ denotes the geodesic distance on $S^n$;
 $\; G=SO(n+1)$; $ \;
K=SO(n)$ and $K^\prime=SO(k+1) \times SO(n-k)$ are  stabilizers
of $e_{n+1}$ and $\, \xi_0$ respectively. The set   $\Xi$ of all $k$-dimensional totally
geodesic submanifolds $\xi$ of $X=S^n$ can be identified with
 the  Grassmann manifold $G_{n+1, k+1}=G/K^\prime$ of all
$(k+1)$-dimensional linear subspaces of $\bbr^{n+1}$. The $G$-invariant probability
measure $d\xi$ on $\Xi$  is defined in a canonical way.

 The  totally geodesic Radon transform $(R f)(\xi)$ of a sufficiently good function $f$ on $S^n$
 is defined by
\be
  (R f)(\xi)=\intl_{\xi} f(x) \, d_\xi x, \qquad \xi \in \Xi,\ee
where  $d_\xi x$ stands for  the usual Lebesgue  measure on $\xi$. Clearly, if $f$ is an odd function, then $Rf\equiv 0$.
We will be dealing with $L^p_e(X)=L^p_e(S^n)$, the subspace of even functions in $L^p(S^n)$.

Another important object  is the shifted dual
Radon transform,  which averages a function $\vp$ on $\Xi$ over all $k$-geodesics $\xi$ at  a fixed distance $\theta$ from $x\in S^n$.
To define this operator, we denote by $r_x\in SO(n+1)$ an
 arbitrary rotation satisfying $r_x e_{n+1}=x$  and set $\vp_x(\xi)=\vp(r_x \xi)$.
 For $\theta \in [0, \pi/2]$,
let $g_\theta$ be the rotation in the plane
$(e_{k+1}, e_{n+1})$ with
the matrix $\left[\begin{array} {cc} \sin\theta &\cos\theta\\ -\cos\theta
&\sin \theta
 \end{array} \right]$.

For $r=\cos\, \theta$, the shifted dual
Radon transform of a function $\vp$ on $\Xi$ is defined by
\be\label {ufus}
(R^*_{x} \vp )(r)= \intl_{d(x, \xi) = \theta} \varphi(\xi) \,
d\mu(\xi)\equiv \intl_{K} \vp_x (\rho
g_\theta^{-1} \xi_0) \,d\rho.\ee
The case $\theta=0$ corresponds to the usual dual
Radon transform \cite{He}. This definition is slightly different from the similar one for the hyperbolic space, however, it allows us to avoid unnecessary technicalities.
\begin{lemma} \label{aawwde} \cite[p. 479]{Ru02b} For all $1\le p \le \infty$,
$$
 \| Rf\|_{(p)} \le \sigma_n^{-1/p} \|
f\|_p, \qquad \| (R^*_{(\cdot)}  \vp)(r)\|_p \le \sigma_n^{1/p} \| \vp \|_{(p)},
$$
where $ \| \cdot \|_{(p)}$ and $\| \cdot \|_{p}$ denote the
 $L^p$-norms on $\Xi$ and $X=S^n$, respectively.
\end{lemma}

 We  need one more averaging operator:
 \be\label {2.21}
 ({\bbm}_x f)(s) = {(1-s^2)^{(1-n)/2}\over \sigma_{n-1}}
 \intl_{\{y \in \bbs^n:\; x \cdot y=s\}} f(y) \,d\sigma  (y), \quad  s \in (-1, 1).\ee
  The integral (\ref{2.21})  is the
mean value of
$f$ on the planar section of $S^n$ by the hyperplane $x \cdot y=s$, and
$ \; d\sigma  (y)$
 stands for the induced Lebesgue measure on this section. It is known that $||({\bbm}_{(\cdot)} f)(s)||_p\le ||f||_p$ and
$\lim\limits_{s\to 1}({\bbm}_x f)(s)=f(x)$ in the $L^p$-norm for all $1\le p\le \infty$.\footnote{Here and on we identify $L^\infty(S^n)$
 with the space  $C(S^n)$  of continuous functions.}

  \begin{lemma}  \label{hhffvb} Let $f\in L^p_e (S^n)$, $1\le p \le \infty$.
 Then
\bea
(R^*_x Rf)(r)&=&\frac{2\sig_{k-1}}{r^{k-1}}\int_0^r (r^2-s^2)^{k/2 -1} ({\bbm}_
{x} f)(s)\, ds\nonumber\\
\label{eeqq1hes}&=& \frac{2\pi^{k/2}}{r^{k-1}}\, (I^{k/2}_{+,2} g_x)(r), \quad g_x(s)=s^{-1}(\bbm_x f)(s).
 \eea
\end{lemma}
This statement can be found  in \cite [p. 140]{He} and \cite [p. 485]{Ru02b}.  Note that, by Lemma \ref {aawwde},  $(R^*_x R)(r)$ represents a bounded operator on $L^p$  for every $ r \in (0, 1)$. Hence, the  integral
 (\ref{eeqq1hes}) is absolutely convergent  for every $r\in (0,1)$ and represents an $L^p$-function of $x$.

Lemma  \ref{hhffvb} implies the following inversion result.

\begin{theorem}\label{invrhys}  Let $X=S^n$.
A function $f \in L^p_e (X)$, $1\le p<\infty$,  can be recovered from $\vp=Rf$ by the formula
\be\label{90ashel}
f(x) \!=  \! \lim\limits_{s\to 1}  \left (\frac {1}{2s}\,\frac {\partial}{\partial s}\right )^k \left [\frac{\pi^{-k/2}}{\Gam (k/2)}\intl_0^s
(s^2 \!- \!r^2)^{k/2-1} \,(R^*_{x} \vp)(r) \,r^k\,dr\right ].\ee
In particular, for $k$ even,
\be\label{90ashele}
f(x) \!=  \! \lim\limits_{s\to 1} \frac{1}{2\pi^{k/2}} \left (\frac {1}{2s}\,\frac {\partial}{\partial s}\right )^{k/2}[s^{k-1}(R^*_{x} \vp) (s)].\ee
Altenatively,
\be\label{90ashys}
f(x) \!=  \! \lim\limits_{s\to 1} \, \left (\frac {\partial}{\partial s}\right )^k \left [\frac{2^{-k}\, \pi^{-k/2}}{\Gam (k/2)}\intl_0^s
(s^2 \!- \!r^2)^{k/2-1} (R^*_{x} \vp)(r) \,dr\right ].\ee
The  limit in these formulas is understood in the $L^p$-norm. If $f\in C_e(X)$, it can be interpreted in the $\sup$-norm.
\end {theorem}
\begin{proof}
By (\ref{eeqq1hes}),
$$(I^{k/2}_{+,2} g_x)(r)= \psi_x (r), \qquad \psi_x (r)=2^{-1} \pi^{-k/2} r^{k-1}(R^*_{x}\vp) (r).$$
Hence, by the semigroup property (\ref{sg-1}), $I^{k}_{+,2} g_x=I^{k/2}_{+,2} \psi_x$, and
therefore, $g_x=D^k I^{k/2}_{+,2} \psi_x$, $D=(1/2s)\,(\partial/\partial s)$. This gives the first two formulas.
Furthermore, by  (\ref{nzxceqq}),
\[(M_x f)(s)=  2^{-k-1} \pi^{-k/2} \left (\frac {\partial}{\partial s}\right )^k \,  (I^{k/2}_{+,2} R^*_{x} \vp)(s),\]
and the third formula follows.
\end{proof}

We observe that the first two formulas in Theorem \ref{invrhys} are well known for infinitely differentiable functions;
cf.  Theorem 1.22 in \cite [p. 141]{He}\footnote {Coincidence of the constants in both theorems follows by duplication formula for gamma functions.}. The third formula, containing usual derivative $(\partial/\partial s)^k$, is new.

\section{On Helgason's formula. Open problem}

The following interesting result is due to Helgason \cite[p. 116]{He}.
\begin{theorem}\label{invrzz} If $k$ is even, then  the $k$-plane transform on $\bbr^n$  can be inverted by the formula
\be\label{lloon} f(x)=c \left [\partial_r^k (R^*_x Rf)(r)\right ]_{r=0}, \qquad c=const,  \qquad  \partial_r=\frac {\partial}{\partial r}\,. \ee
\end{theorem}

This formula is much simpler than those in Theorem \ref{invrhys}.
 The constant $c$ in (\ref{lloon}) was explicitly evaluated in \cite{AR}, where Theorem \ref{invrzz} has been extended to totally geodesic Radon transforms on arbitrary constant curvature space $X$. To state this result, we
introduce the {\it distance function} \be\label {mmi}
\rho(x,\xi)=\left\{ \!
 \begin{array} {ll} d(x,\xi) & \mbox{if $X=\bbr^n $,}\\
  \sinh \,d(x,\xi) & \mbox{if $X=\bbh^n $,}\\
  \sin d(x,\xi) & \mbox{if $X=S^n $.}\\

  \end{array}
\right. \ee The  corresponding shifted dual Radon transform can be defined by
\be\label {aaee}
(R^*_x\vp)(r)=\intl_{\rho(x,\xi)=r} \vp(\xi) \, d \mu (\xi), \qquad x \in X, \quad r>0,\ee
where $d \mu (\xi)$ is the relevant normalized canonical measure.

 \begin{theorem} \label {koom2} Let $\vp=R f$.  If $k$ is even, then
  \be\label {hel1}
  \partial_r^k \lam_X (r)(R^*_x \vp)(r)\Big |_{r=0}=c_X\,f(x),
\ee
where
 \[
\lam_X (r)=\left\{ \!
 \begin{array} {ll} 1 & \mbox{if $\;X=\bbr^n $,}\\
  (1+r^2)^{(k-1)/2} & \mbox{if $\;X=\bbh^n $,}\\
   (1-r^2)^{(k-1)/2} & \mbox{if $\;X=S^n $,}\\
 \end{array}
\right. \]

\[c_X=\left\{ \!
 \begin{array} {ll} (-1)^{k/2}(k-1)!\, \sig_{k-1} & \mbox{if $\;X=\bbr^n, \bbh^n $,}\\
  2(-1)^{k/2}(k-1)!\, \sig_{k-1} & \mbox{if $\;X=S^n $.}\\
   \end{array}
\right.\]
\end {theorem}

In both theorems it was assumed that $f$ is infinitely smooth and (for $X=\bbr^n, \bbh^n $) rapidly decreasing.
This assumption is redundant. See Appendix, where, for $X=\bbr^n$, it is shown that the result holds  under much weaker asumptions.

\vskip 0.3truecm

\noindent{\bf Open Problem.} {\it Extend Theorems  \ref{invrzz} and  \ref{koom2} to non-differentiable functions, e.g., $f\in L^p(X)$.}

Helgason's idea to invoke usual differentiation in place of $d/dr^2$ agrees with the 1927  paper by Mader \cite {Mad} in the sense that
  her inversion formula for the hyperplane Radon transform   also contains usual differentiation. However, the method of \cite {Mad} is  completely different.  For the sake of completeness, we present without proof a generalization of Mader's result, which was obtained in \cite{AR}.

For $r>0$ and $1 \le k \le n-1$, let
\be\label {las8}
 (L^*_x \vp)(r)=\intl_{\Xi}
  \vp (\xi)\, \rho^{k+1-n} \,\sgn (\rho-r) \,d\xi, \ee
\be
 (\tilde L^*_x \vp)(r)=\intl_{\Xi}
  \vp (\xi)\, \rho^{k+1-n} \,\log |\rho^2-r^2| \,d\xi,\ee
where $\rho= \rho(x, \xi)$ is the distance function  (\ref{mmi}).
  \begin{theorem} \label {koom} Let $\vp=R f$, where $f$ is a $C^\infty$ function, which is rapidly decreasing in the case $X=\bbr^n, \bbh^n $.

\vskip 0.3truecm

 \noindent {\rm (i)} If $k$ is even, then
   \be\label {hel2m}
\partial_r^{k+1} (L^*_x \vp)(r)\Big |_{r=0}= d_X\,f(x),
\ee
 where
 \[d_X=\left\{ \!
 \begin{array} {ll} 2(-1)^{(k+2)/2}\sig_{n-k-1} \sig_{k-1}\,(k-1)!  & \mbox{if $\;X=\bbr^n, \bbh^n $,}\\
2\, \sigma_{n-k-1}\, \sigma_k\, \sigma_{k-1}\,(k-1)! /\sigma_n & \mbox{if $\;X=S^n $.}\\
   \end{array}
\right.\]

 \noindent {\rm (ii)} If $k$ is odd, then
   \be\label {hel2}
\partial_r^{k+1} (\tilde L^*_x \vp)(r)\Big |_{r=0}=\tilde d_X\,f(x),
\ee
 where
 \[\tilde d_X=\left\{ \!
 \begin{array} {ll} \pi\,(-1)^{(k-1)/2} \sig_{n-k-1}\, \sig_{k-1} \,(k-1)! & \mbox{if $\;X=\bbr^n, \bbh^n $,}\\
2\pi (-1)^{(k-1)/2}\, \sigma_{n-k-1}\, \sigma_k\, \sigma_{k-1}\,(k-1)! /\sigma_n & \mbox{if $\;X=S^n $.}\\
   \end{array}
\right.\]
\end {theorem}

\section{Appendix}

Let us prove  Theorem \ref{koom2} for the case $X=\bbr^n$. We provide  more details, than in \cite{AR}, and
make assumptions for $f$ more precise.
For $\mu>0$ and $k\in \bbn$, let
$$C^k_\mu (\bbr^n)= \{ f \in C^k (\bbr^n): \ |\partial^\a f (x)=O(|x|^{-\mu})
\;\, \forall \;  |\a| \le k \}.$$

\begin{proposition} Suppose that $1\le k\le n-1$, $\mu>k$. Then
  a function $f\in C^k_\mu (\bbr^n)$  can be reconstructed from  $\vp=Rf$ by the formula
\be\label{kopppp} f(x) =  \lim\limits_{r\to 0}\,c_k\,\left (- \partial_r \right )^{k} \,(R_x^* \, \vp)(r), \qquad
c_k=\frac{(-1)^{k/2}}{(k-1)!\, \sig_{k-1}},\ee
where the  limit  is uniform on $\bbr^n$.
\end{proposition}
\begin{proof} Fix $x$ and  write (\ref{gyuyy}) in the form
  $$(R^*_x \vp) (r)=\sig_{k-1}\, [A(r)+(-1)^{k/2}\,B(r)],$$ where
\[A(r)=\intl_0^\infty \!(\M_x f)(t) (t^2
-r^2)^{k/2 -1} t \, dt, \]
\[ B(r)=\intl_0^r (\M_x f)(t)\, (r^2-t^2)^{k/2 -1} t\,dt. \]
Since $A(r)$ is a polynomial  of degree $k-2$, then $\partial_r^k A(r)=0$.
Regarding $B(r)$, we write it  as $B_1 +  B_2$, where
$$
B_1=f(x)\intl_0^r  (r^2-t^2)^{k/2 -1} t \, dt=\frac{ r^k}{k}\, f(x),
$$
$$
B_2=\intl_0^r  (r^2-t^2)^{k/2 -1} [(\M_t f)(x) - f(x)]\,  t \, dt=r^k h(r),$$
$$
h(r)= \intl_0^1  (1-t^2)^{k/2 -1} [(\M_{rt} f)(x) - f(x)]\,  t \, dt.
$$
Clearly, $\partial_r^k B_1=(k-1)!  \, f(x)$. Furthermore,

\[\lim\limits_{r\to 0} \partial_r^k B_2=\sum\limits_{j=0}^k c_j \lim\limits_{r\to 0} r^j h^{(j)} (r).\]
The term corresponding to $j=0$ is obviously zero. Other terms are also zero because $h^{(j)} (r)$ is uniformly bounded. Indeed,
\bea |h^{(j)} (r)|&\le&\intl_0^1  (1-t^2)^{k/2 -1}  t \, dt \intl_{S^{n-1}} | \partial_r^j  f(x+rt\theta)|\, d\theta \nonumber\\
&\le&\sum\limits_{|\a|=j}\intl_0^1  (1-t^2)^{k/2 -1}  |\,\theta ^\a \partial^\a f (x+rt\theta)|\,t^{j+1}  \, dt\nonumber\\
&\le& \tilde c_j \sup_{|\a|=j} \, \sup_x |\partial^\a f (x)|.\nonumber\eea
Thus, $\lim\limits_{r\to 0}\partial_r^k B_2(r)=0$, and the result follows.
\end{proof}


\bibliographystyle{amsalpha}

\end{document}